\documentclass[12pt]{article}
\usepackage{amssymb,amsmath, amsthm, amscd,ifthen}
\usepackage[T2A]{fontenc}
\usepackage[cp1251]{inputenc}
\usepackage[russian,english]{babel}
\usepackage[dvips]{graphicx}
\usepackage{indentfirst}

\newcommand{\eps}{\varepsilon}

\renewcommand{\le}{\leqslant}
\renewcommand{\ge}{\geqslant}

\renewcommand{\leq}{\leqslant}
\renewcommand{\geq}{\geqslant}

\newcommand{\N}{\mathbb N}

\newcommand{\R}{\mathbb{R}}

\renewcommand{\le}{\leqslant}
\renewcommand{\ge}{\geqslant}
\renewcommand{\phi}{\varphi}

\newcommand{\mk}{\mathcal}

\newtheorem{defn}{Definition}

\newtheorem{prob}{Problem}

\numberwithin{equation}{section}
\newtheorem{thm}{Theorem}[section]
\newtheorem{cor}[thm]{Corollary}
\newtheorem{prop}[thm]{Proposition}

\newtheorem{lem}[thm]{Lemma}

\begin{document}

\title{Two notions of unit distance graphs}
\author{
Noga Alon
\thanks{Sackler School of Mathematics
and Blavatnik School of Computer Science,
Tel Aviv University,
Tel Aviv 69978, Israel.
Email: {\tt nogaa@tau.ac.il}.
Research supported in part by an ERC Advanced
grant, by a USA-Israeli BSF grant,
by the Hermann
Minkowski Minerva Center for Geometry at Tel Aviv University
and by the Israeli I-Core program.}
\and
Andrey Kupavskii
\thanks{
Ecole Polytechnique F\'ed\'erale de Lausanne; Department of Discrete Mathematics, Moscow Institute of Physics and Technology.
Email: {\tt kupavskii@yandex.ru}.
Research supported in part by the grant  N MD-6277.2013.1 of President of RF and by the grant N 12-01-00683 of the Russian Foundation for Basic Research.}
}

\maketitle

\begin{abstract}
A {\em faithful (unit) distance graph} in $\mathbb{R}^d$
is a graph whose set of
vertices is a finite subset of the $d$-dimensional Euclidean space,
where two vertices are adjacent if and only if the Euclidean
distance between them is exactly $1$. A {\em (unit) distance graph} in
$\mathbb{R}^d$ is any subgraph of such a graph.

In the first part of the paper we focus on the differences
between these two classes of graphs. In particular, we show that for
any fixed $d$ the number of faithful distance graphs
in $\mathbb{R}^d$ on $n$
labelled vertices is $2^{(1+o(1)) d n \log_2 n}$, and give a short
proof of the known fact that the number of distance graphs in
$\mathbb{R}^d$ on $n$ labelled vertices is
$2^{(1-1/\lfloor d/2 \rfloor +o(1))n^2/2}$. We also study the behavior of
several Ramsey-type quantities involving these graphs.

In the second part of the paper we discuss the problem
of determining the minimum possible number of edges of a
graph which is not isomorphic
to a faithful distance graph in $\R^d$.
\end{abstract}

\section{Introduction}

\subsection{Background}

We study the differences between
the following two well-known notions of (unit) distance graphs:

\begin{defn}\label{dist1}
{\rm A graph $G=(V,E)$ is a} \textit{(unit) distance graph  in
$\mathbb{R}^d$,} {\rm if $V\subset \mathbb{R}^d$ and $E\subseteq\{(x,y):
x,y \in V, |x-y|=1\}$, where $|x-y|$ denotes the
Euclidean distance between $x$ and $y$.}
\end{defn}
\begin{defn}\label{dist2}
{\rm A graph $G=(V,E)$ is a} \textit{faithful (unit) distance graph
in $\mathbb{R}^d$,} {\rm if $V\subset \mathbb{R}^d$ and $E=\{(x,y): x,y \in
V, |x-y|=1\}.$}
\end{defn}
We say that a graph $G$ is realized as a (faithful) distance graph in $\R^d$, if it is isomorphic to some (faithful) distance graph in $\R^d$.
Denote by $\mathcal D(d)$ ($\mathcal D_n(d)$) the set of all labelled
distance graphs in $\R^d$ (of order $n$). Similarly, denote by $\mathcal{FD}(d)$ ($\mathcal{FD}_n(d)$)
the set of all labelled faithful distance graphs in
$\R^d$ (of order $n$).

Distance graphs appear in the investigation of two well-studied
problems.
The first is the problem of determining the chromatic number $\chi(\R^d)$
of the $d$-dimensional space:

\begin{equation*}
\chi(\mathbb{R}^d)=\min\{m\in\mathbb{N}:
\mathbb{R}^d = H_1\cup \ldots\cup H_m:
  \forall i, \forall x,y \in H_i,\ \ |x-y|\neq 1\}.\notag
\end{equation*}

\noindent The second is the investigation of the maximum
possible number $f_2(n)$ of
pairs of points at unit distance apart in a set of $n$ points
in the plane $\R^2$.
Distance graphs arise naturally in the context of both
problems. Indeed,
\begin{align*}
\chi(\R^d)&=\max_{G\in \mathcal D(d)}
\chi (G) = \max_{G\in\mathcal{FD}(d)}\chi( G),\\
f_2(n)&=\max_{G\in \mathcal D_n(2)} |E(G)| =
\max_{G\in\mathcal{FD}_n(2)} |E(G)|.
\end{align*}

\noindent Thus, in the study of these two extremal problems it does not
matter whether we consider distance graphs or faithful distance graphs.
However, there is a substantial difference between the sets $\mathcal D(d)$
and $\mathcal {FD}(d)$. This difference is discussed in the
theorems that appear in
what follows.

\subsection{The main results}
The first theorem provides some classes of graphs that
are (or are not) distance or faithful distance graphs in $\R^d$.
A surprising aspect of Theorem \ref{th1} is that for any $d$ there are
bipartite graphs that are not faithful distance graphs in $\R^d$.
\begin{thm}
\label{th1}

1. Any $d$-colorable graph can be realized as a distance graph in
$\R^{2d}$.

2. Let $d\in \N$ and $d\ge 4.$ Consider the graph $K'=K_{d,d}- H,$ where
$H$ is a matching of size $d-3.$ Then the graph $K'$ is not realizable
as a faithful distance graph in $\R^d$.

3. Any bipartite graph with maximum degree at most $d$ in one of its parts
so that no three vertices of degree $d$ in this part have exactly the same
set of neighbors
is realizable as a faithful distance graph in $\R^d$.
\end{thm}

The next theorem shows that in any dimension $d$
there are far more distance graphs than
faithful distance graphs.

\begin{thm}
\label{th2}
1. For any $n, d\in \N, n\ge 2d,$
we have $|\mathcal {FD}_n(d)|
\leq {{n(n-1)} \choose {nd}}.$
Therefore, for any fixed $d$, $|\mathcal {FD}_n(d)|=
2^{(1+o(1)) dn \log_2 n}$.

2. (A. Kupavskii, A. Raigorodskii, M. Titova, \cite{KRT}). For any fixed $d \in \N$ we have $|\mathcal D_n(d)|=2^{\left(1-\frac 1{[d/2]}
+ o(1)\right)\frac {n^2}2}.$

\end{thm}

By simple calculations one can obtain the
following corollary from the upper bound in part 1 of Theorem \ref{th2}:

\begin{cor}\label{cord}
1. If $d=d(n)=o(n),$ then we have $|\mathcal {FD}_n(d)|=2^{
o(n^2)}.$

2. If $d = d(n)\le cn,$ where $0<c<1/2$ and $H(c)
<1/2$ (here $H(z)=-z \log_2 z -(1-z) \log_2 (1-z) $ is the binary
entropy function), then there exists a constant
$c'=c'(c)<1/2$ such that $|\mathcal {FD}_n(d)|\le 2^{c'n^2(1+o(1))}.$
\end{cor}

As  proved by B. Bollob\'as \cite{Bol}, with high
probability the random graph $G(n,1/2)$ has chromatic number $(1+o(1))\frac
n{2\log_2 n}$. By part 1 of Theorem \ref{th1}, any $k$-colorable graph is realizable as a distance graph in $\R^{2k}$.
It means that if $d = d(n)\ge c\frac n{\log_2 n}, $ where $c>1$, then $|\mathcal
D_n(d)|=(1+o(1)) 2^{\frac {n(n-1)}2}.$
In other words, for such $d$ almost every
graph on $n$ vertices can be realized as a distance graph in $\R^d$.
This is very different from the behaviour of faithful distance
graphs, as shown in Corollary \ref{cord}.\\
\vspace{0.3cm}

We next consider the following extremal problem.

\begin{prob} Determine
the minimum possible number $g(d)$ of edges of a graph $G$ which is not
realizable as a faithful distance graph in $\R^d$.
\end{prob}

An intriguing question here is whether or not for any $d\ge 4$, $~g(d)= {d+2\choose 2}$.
In other words, does $K_{d+2}$ have the minimal number of edges
among the graphs that are not realizable as faithful distance graphs
in $\R^d,\ d\ge 4$. Interestingly, this is not the case
in $\R^3$, since the graph
$K_{3,3}$ is not realizable as a faithful distance graph
in $\R^3$ and it has fewer edges than $K_5$.

We restrict our attention here to bipartite graphs, studying the
following problem.

\begin{prob}\label{pr1}
Determine the minimum possible number
$g_2(d)$ of edges of a bipartite graph $K$ which is
not realizable as a faithful distance graph in $\R^d$.
\end{prob}

Note that the minimum number of vertices
such a $K$ can have equals $2d$, as
follows from parts 2 and 3 of Theorem \ref{th1}.

\begin{thm}\label{bipart} For any $d\ge 4$
we have  ${d+2\choose 2}\le g_2(d)\le {d+3\choose 2}-6.$
\end{thm}

\vspace{0.3cm}

\noindent
\textbf{Remark.}

After completing this manuscript we learned that some of  the
questions discussed here have already been studied by
Erd\H os and Simonovits in \cite{ErdSim} and by  Maehara in
\cite{Mae}.
It seems that
Maehara was unaware of the paper \cite{ErdSim}. We proceed with a
brief comparison between the
the results of these two papers and our results here.
Part 3 of Theorem \ref{th1} slightly improves the bipartite case of Theorem
2 from \cite{Mae}, which states the following: if a graph $G$ has maximum degree $d$ and $\chi(G)=k$, then $G$ is faithful distance in $\R^D$, where $D =  {k\choose 2}(d+1)$. In part 2 of Theorem  \ref{th1} we present a
graph which is not realizable as a faithful distance
graph in $\R^d$. Constructions of such graphs can be found in both
papers \cite{ErdSim} and \cite{Mae}. The construction of Erd\H os
and Simonovits (given in Proposition 1 of \cite{ErdSim}) is similar to the
construction we use, however, it is slightly worse in terms of the
number of
vertices and edges (the smallest known construction, which is used in the
proof of the upper bound in Theorem \ref{bipart}, is a bipartite graph with parts $A=\{a_1,\ldots,
a_d\}, B=\{b_1,\ldots,b_d\}$ and the set of edges $E=\{(a_i,b_j):
i>j\}\cup\{(a_i,b_j):i\le 3\}$).
The graph used by Maehara is much
bigger than both graphs used by us and by Erd\H os and Simonovits. Note
that our graph is in some sense best possible, as follows from
the assertion of
part 3 of Theorem \ref{th1}.

The main results of both
papers \cite{Mae} and \cite{RM} establish bounds on the dimension in
which a graph can be realized as a faithful distance graph in terms of
the maximum degree and the chromatic number (see Theorem \ref{thRM} in
the present paper).
Similar bounds were already proved by Erd\H os and
Simonovits (see Theorem 6 in \cite{ErdSim}), and their bound differs from
the bound of R\" odl and Maehara only by 1 (Erd\H os and Simonovits prove
that any graph with maximum degree $k$ can be realized as a faithful
distance graph in $\R^{2k+1}$, while R\" odl and Maehara prove that
such graphs can be realized in  $\R^{2k}$, the proofs rely on
similar constructions).

A question analogous to Problem \ref{pr1}, for distance graphs
instead of faithful distance graphs, was asked in \cite{ErdSim} (problem
5 in \cite{ErdSim}).

\subsection{More on the difference
between distance and faithful distance graphs}

We study two Ramsey-type quantities related to
distance and faithful distance graphs.
\begin{defn}\label{rams1}
The \textit{(faithful) distance Ramsey number} $ R_{D}(s,t,d) $
$\bigl(R_{FD}(s,t,d)\bigr)$ is the minimum integer
$m$ such that for any graph $G$ on $m$ vertices the following holds:
either $G$ contains an induced $s$-vertex subgraph isomorphic to a (faithful)
distance graph in $ {\mathbb{R}}^d $ or its complement $ \bar {G} $
contains an induced $t$-vertex subgraph isomorphic to a (faithful) distance graph
in ${\mathbb{R}}^d.$
\end{defn}

The quantity $R_D(s,s,d)$ is studied in \cite{KRT}, where the
following theorem is proved:
\begin{thm}[A. Kupavskii, A. Raigorodskii, M. Titova \cite{KRT}]
\label{thKRT}
1. For every fixed $d\in \mathbb N$ greater than 2 we have

$$
R_{D}(s,s,d) \geq 2^{\left(\frac{1}{2[d/2]}+ o(1)\right)s}.
$$

2. For any  $d=d(s)$, where $ 2\le d \le s/2,$ we have

$$
R_{D}(s,s,d) \leq  d \cdot R\left(\left\lceil\frac{s}{[d/2]}
\right\rceil,\left\lceil\frac{s}{[d/2]}\right\rceil\right),
$$
where $R(k,\ell)$ is the classical Ramsey number, which is the minimum
number $n$ so that any graph on $n$ vertices
contains either a clique of size $k$ or an independent set of size
$\ell$.
\end{thm}

Note that it is in fact not difficult to improve the upper bound to
$$
R\left(\left\lceil\frac{s}{[d/2]}
\right\rceil,\left\lceil\frac{s}{[d/2]}\right\rceil\right) +2s,
$$
but for our purpose here this improvement is not essential and we
thus  do not include its proof.

By the previous theorem the bounds for $R_{D}(s,s,d)$
are roughly the same as for the classical Ramsey number
$R\left(\left\lceil\frac{s}{[d/2]}
\right\rceil,\left\lceil\frac{s}{[d/2]}\right\rceil\right)$:

\begin{equation*}
\frac {s}{2[d/2]}(1+o(1))\le
\log R_{D}(s,s,d) \le \frac {2s}{[d/2]}(1+ o(1)),
\end{equation*}
where the $o(1)$-terms tend to zero as $s$ tends to infinity.

What can we say about $R_{FD}(s,s,d)$?
It turns out that $R_{FD}(s,s,d)$ is far larger  than
$R_D(s,s,d)$. Using Theorem \ref{th2} we can prove the following
result:

\begin{prop}\label{thram}
1. For any $d=o(s)$ we have $R_{FD}(s,s,d)\ge 2^{(1+
o(1))s/2}$.

2. For $d\le cs,$ where $c<1/2$ and $H(c)<1/2,$
there exists a constant $\alpha=\alpha(c)>0$ such that
$R_{FD}(s,s,d)\ge 2^{(1+ o(1))\alpha s}.$
\end{prop}

It is worth mentioning that, if $d=cs$ for a sufficiently small $c>0$,
the quantity $R_{FD}(s,s,d)$ grows exponentially, while the quantity
$R_{D}(s,s,d)$ grows linearly (this follows from part
2 of Theorem \ref{thKRT}).\\

The final (possible) difference
between $\mathcal D(d)$ and $\mathcal {FD}(d)$
we point out is the following. Fix an $l\in \N$ and
consider the distance graphs from $\mathcal D(d)$ and $\mathcal {FD}(d)$ that have girth greater than $l$. What can we say about the chromatic number of such graphs?

\begin{thm}[A. Kupavskii, \cite{Kup}]
For any $g\in \mathbb{N}$ there exists a sequence of
distance graphs in
$\mathbb{R}^d,\ d=1,2,\ldots$ with girth greater than $g,$ such that the
chromatic number of the graphs
in the sequence grows exponentially in  $d$.
\end{thm}

In the analogous problem for faithful distance
graphs the situation is less understood.
All we can prove here is the following

\begin{prop}
\label{prgir}
For any $g\in \mathbb{N}$ there exists a sequence of faithful
distance graphs in
$\mathbb{R}^d,\ d=1,2,\ldots,$ with girth greater than $g$ such that the
chromatic number of the graphs in the sequence grows as  $\Omega_g \left(\frac
d{\log d}\right)$.
\end{prop}

\section{Proofs}
\subsection{Proof of Theorem  \ref{th1}}

\textbf{1.} Two circles are orthogonal
if they lie in orthogonal two-dimensional planes.  Choose $d$ pairwise
orthogonal circles of radius $1/\sqrt 2$ with a common center. The distance between any two points from different circles equals 1. Embed each color class into one circle.

\textbf{2.} Suppose the graph $K'$ can be realized as a faithful distance
graph in $\R^d$. Denote both the vertices of $K'$ and the points in the space
that correspond to them by the same letters. Let
$A=\{a_i\},B=\{b_i\}$ be the parts of $K'$, where $|A|=|B|=d,$ and the edges
$(a_i,b_i)$, where $i\in \{4,\ldots,d\},$ are not present in the graph. Then $a_{1},a_{2},
a_{3}$
are the vertices that are connected to all vertices of
$B$. In any faithful distance realization of $K'$ the vertices $a_{1},a_{2}, a_{3}$
must be affinely independent. Indeed, they cannot lie on the same line
since otherwise there will be no points at unit distance
from all of them.

The points $b_j$ must lie on the $(d-3)$-dimensional sphere $S$ in the $(d-2)$-dimensional
subspace, orthogonal to the plane containing $a_{1},a_{2}, a_{3}$. The
sphere $S$ has the same center as the circle, circumscribed around the
triangle $a_{1},a_{2}, a_{3}$.

Now we show that $b_i$ is affinely independent of the
points $b_j,j< i$. For
$i\le 3$ this is clear for the same reason as for the points $a_1,a_2,a_3$.
For $i\ge 4$, consider the sphere $S_i=S\cap \mathrm{aff}\{b_j, j<
i\},$ where by $\mathrm{aff}\{x_1,\ldots, x_l\}$ we denote the affine hull of the points $x_1,\ldots, x_l$. The sphere $S_i$ is contained in the sphere with the center in $a_i$
and unit radius, because all points $b_j,j<i$, are connected
to $a_i$. But if the point $b_i$ lies in $\mathrm{aff}\{b_j, j< i\}$,
then $b_i\in S_i$. Thus, we are forced to draw an edge $(a_i,b_i)$,
which is forbidden.

Since each $b_i$ is affinely independent of $b_j,j< i$, we obtain $d$
affinely independent points in $\R^{d-2}$ --- a contradiction.

\textbf{3.} Let $K$ be an arbitrary bipartite graph with parts $A=\{a_i\},
B=\{b_j\}$ satisfying the condition $\max_{i} \deg (a_i)\le d,$ and such that no three vertices from $A$ of degree $d$ have the same set of neighbors.
We introduce the following notation:  a sphere $S'$
is \textit{complimentary} to the sphere $S^f$ of dimension $f\le (d-2)$
in the space $\R^d$ if $S'$ is formed by all points of $\R^d$ that are
at unit distance apart from the points of $S^f$. For a set of points $X$ we use the notation $S(X)$ for the sphere of minimal dimension that contains all points from $X$ (if one exists), and $S'(X)$ for the sphere, complimentary to $S(X)$ (again, if one exists).

We realize $K$ as a faithful distance graph in $\R^d$.
First embed all points of $B$ in the space $\R^d$ so that the diameter of
$B$ is smaller than 1 and all points lie in a sufficiently general position:
{\bf (a) } No $k$
points of $B$ lie in a $(k-2)$-dimensional plane , $k=1,\ldots, d$.\\
{\bf (b) } No $d+1$ points lie on a unit sphere.\\
{\bf (c) } There are no two subsets $B_1,B_2$ of $B$, both of size $d$, such that the distance between some of the points of $S'(B_1), S'(B_2)$ is 1 (note that $S'(B_i)$ consists of two points).\\
{\bf (d) } There are no two subsets $B_1,B_2$ of $B$, such that  $S'(B_1)\subset S(B_2)$. Moreover, if $B_1$ is of size $d$, then $S'(B_1)\cap S(B_2)=\emptyset$.\\
All the forbidden positions of the points from $B$ may be expressed as zero sets of certain polynomials, so we can avoid all of them.

Next we embed the set $A$. Each point $a_i$ is
connected to $n_i\le d$ points from $B$. Denote this set by $B_i$.
By (a) the points
from $B_i$ form a $(n_i-1)$-dimensional simplex  with circumscribed
sphere $S$ of radius $r<1/\sqrt 2$. Consider a sphere $S'(B_i)$. The
dimension of  $S'(B_i)$ is $d-n_i\ge 0$.

First we embed all the points of $A$ that have $d$ neighbors in
$B$ one by one. For each such point $a_i$ there are two possible points
in $\R^d$  with which  it can coincide and at most one of them is already
occupied. Condition (b) guarantees that we do not get any extra edges between $a$ and points not from $B_i$. Condition (c) guarantees that we cannot get an edge between the vertices $a_i,a_j\in A$ of degree $d$.

  The remaining points of $A$ can now be
embedded, one by one, in the following way.
We embed the point $a_i$ onto $S'$ in
such a way that the distance between $a_i$ and the points from $B\backslash
B_i$  is not unit, $a_i$ does not
 coincide or at unit distance apart from any previously placed $a_j$ and $a_i$ does not fall into $S(B_l)$ for any $l$.
This can be done since any sphere of unit radius with center in any of the points from $B\backslash B_i$ can intersect
$S'(B_i)$ only by a sphere $S''\subset S'(B_i)$ of smaller dimension due to (a), and the same holds for spheres with centers in $a_j$ due to the fact that no $a_j$ fall into $S(B_l)$. This is, in turn, possible due to (d), out of which we get that for any $k,l$ the sphere $S'(B_k)\cap S(B_l)$ is a sphere of strictly smaller dimension than $S'(B_k)$.\\
\vspace{0.3cm}

\noindent
{\bf Remark. } Condition (b) can be satisfied just by choosing points
on a sphere of radius smaller than 1, and conditions (c), (d) can be
satisfied by additionally requiring that for some small $\eps$ all the
points are $\eps$-flat with respect to some hyperplane $\gamma$,
that is, all the hyperplanes determined by points of the set $B$ form an
angle with $\gamma$ which is smaller than $\eps$ and all the pairwise
distances between the points are at most $\eps r,$ where $r$ is the
radius of the sphere on which the points lie. For (c) we additionally
require that $r$ is not close to $1/2$ in terms of $\eps$. This
will be used
use in the proof of Theorem \ref{bipart}.\\

\subsection{Proof of Theorem \ref{th2}}

\textbf{1.} Let $P_1,\ldots, P_m$ be $m$ real polynomials in $l$
real variables. For a point $x \in \R^l$ the \textit{zero pattern}
of the $P_j$'s at $x$ is the tuple $(\eps_1,\ldots, \eps_m)\in
\{0,1\}^m,$ where $\eps_j = 0$, if $P_j(x)=0$ and $\eps_j=1$
if $P_j(x)\ne 0$.  Denote by $z(P_1,\ldots, P_m)$ the total
number of zero patterns of the polynomials $P_1,\ldots, P_m$.

The upper bound in part 1 of the theorem is a corollary of
the following proposition~(\cite[Theorem 1.3]{Ron}):
\begin{prop}[L. R\'onyai, L. Babai, M.K. Ganapathy  \cite{Ron}]
\label{prop1}
Let $P_1,\ldots, P_m$ be $m$ real
polynomials in $l$ real variables, $m\ge l$, and suppose the degree of each
$P_j$ does not exceed $k$. Then $z(P_1,\ldots, P_m)
\le {{km - (k-2)\ell} \choose \ell}.$

\end{prop}

Associate a family of $n(n-1)/2$ polynomials $P_{ij}$
in $dn$ real variables with an arbitrary labelled
distance graph $G$ of order
$n$ in $\R^d$ as follows.
Denote by $(v^i_1,\ldots, v^i_d)$ the coordinates of the
vertex $v_i$ in the distance graph. For each
unordered
pair $\{i,j\}$ of vertices of the graph define a polynomial
$P_{ij}$ that
corresponds to the square of the distance between the pair $v_i, v_j$
minus $1$:

        $$P_{ij} = -1+\sum_{r=1}^d(v^i_r-v^j_r)^2.$$

It is easy to see that each labelled distance graph in $\R^d$
corresponds
to a zero pattern of the polynomials $P_{12},\ldots,P_{(n-1)n}.$  It is also clear
that different distance graphs correspond to different zero patterns,
since this pattern specifies the set of (labelled) edges in the
graph. Thus the number of labelled faithful distance graphs of order
$n$  in $\R^d$
is at most the number of zero patterns of the above
polynomials. All
we are left to do in order to establish the upper bound  in part
1 is to substitute $k=2, l=dn, m=\frac {n(n-1)}2$ in
Proposition \ref{prop1}.

The lower bound follows from part 3 of
Theorem \ref{th1} by taking a bipartite graph with classes of
vertices $B$ of size, say, $n/ \log n$ and $A$ of size $n-|B|$, so
that any vertex of $A$ has exactly $d$ neighbors in $B$  and no two
vertices of $A$ have exactly the same set of neighbors. This
implies that
$$
|\mathcal {FD}_n(d)| \geq |A|!{{|B| \choose d} \choose {|A|}},
$$
supplying the desired asymptotic bound.

\textbf{2.} This was proved in \cite{KRT}. Here we present a short proof
of this fact using the following proposition from \cite{KRT}
and a theorem from \cite{EFR}:

\begin{prop}[A. Kupavskii, A. Raigorodskii, M. Titova \cite{KRT}]
The graph $K_{\underbrace{3,\ldots,3}_{[d/2]+1}}$ is not
realizable as a distance graph in $\R^d$.
\end{prop}


\begin{thm}[P. Erd\H os, P. Frankl, V. R\"odl, \cite{EFR}]
Let $G$ be a graph, $\chi(G) =r\ge 3.$ Then the number
$F_n(G)$ of labelled graphs of order $n$, not containing a copy of $G$,
satisfies: $F_n(G) = 2^{\left(1-\frac 1{r-1}+
o(1)\right)\frac{n^2}2}.$
\end{thm}

Applying the above we get
$$
|D_n(d)|\le
F_n\biggl(K_{\underbrace{3,\ldots,3}_{[d/2]+1}}\biggr)=2^{\left(1-\frac
1{[d/2]}+ o(1)\right)\frac{n^2}2}.
$$

On the other hand, by part
1 of Theorem~\ref{th1}
$$|D_n(d)|\ge 2^{\left(1-\frac 1{[d/2]}+
o(1)\right)\frac{n^2}2},
$$
since any $[d/2]$-partite graph of order $n$
is realizable as a distance graph in $\R^d$.

\subsection{Proof of Theorem \ref{bipart}}
We first prove the upper bound. It is easy to check that
 in the proof of
part 2 of Theorem \ref{th1} the edges $(a_i,b_j),
i,j> 3, i<j$ are not used,
and thus their presence or absence does not affect the validity of
the proof. In
particular, the bipartite graph $K''$ with parts $A=\{a_1,\ldots,
a_d\}, B=\{b_1,\ldots,b_d\}$ and the set of edges $E=\{(a_i,b_j):
i>j\}\cup\{(a_i,b_j):i\le 3\}$ is not realizable as a faithful
distance graph in $\R^d$. The number of edges in this graph is
${d+3\choose 2}-6,$ establishing
the bound $g_2(d)\le {d+3\choose 2}-6.$
\vskip+0.2cm

\textbf{Remark.} The following bipartite
graph is also not realizable as a faithful
distance graph in $\R^d$: $H=(A\cup B, E),\ A=\{a_1,\ldots, a_{d+2}\},
B=\{b_1,\ldots,b_{d+2}\},\ E=\{(a_i,b_j): i\ge j\}.$ This graph has
${d+3\choose 2}$ edges.\vskip+0.2cm

We proceed with the proof of the lower bound.
Consider a bipartite graph $G = (A\cup B, E)$, where $A = \{a_1,\ldots,
a_{n+s}\},\ B =\{b_1,\ldots, b_m\}.$ Suppose that the vertices of $A$
are ordered in such a way that $\deg(a_i)\le \deg(a_j)$ if  $ i<j.$
Suppose also that $\deg(a_{n+1})=\deg(a_{n+s}) = m, \deg(a_n)<m.$ We
provide sufficient conditions for $G$ to be realizable as a faithful
distance graph in $\R^d.$ Recall the following notation: a sphere $S'$
is \textit{complimentary} to a sphere $S^f$ of dimension $f\le (d-2)$
in the space $\R^d$, if $S'$ is formed by all points of $\R^d$ that are
at unit distance apart from the points of $S^f$.\\

Here is an outline of the proof.  The general goal is to find a good
realization for the set $B$. Having such a realization,
the vertices of $A$ will be placed on the corresponding complimentary
spheres using a general position argument, in a similar way to that
described in the proof of part 3
of Theorem \ref{th1}. To find the realization, we want the points
of $B$
to satisfy the analogues of the conditions (a), (b), (c), (d). Conditions
(b), (c), (d) are technical and can be satisfied without difficulties
in this case (see the remark following the proof of Theorem \ref{th1}). The
main difference is concerning condition (a): we cannot simply place all
vertices of $B$ in a sufficiently general position, since there may
be vertices in $A$ of a very high degree, and we may get unexpected edges.

In most cases we try to place the vertices of $B$ on a sphere,
and, as we have already seen in
the proof of part 2 of Theorem \ref{th1}, if
the vertices of
a part of a bipartite graph lie on a sphere then there is
a tight connection between the presence of certain edges and the affine
independence of certain vertices.

We begin with all points of $B$ on the circle and start to modify the
realization so that it is getting closer and closer to the desired one. To
be more precise, we treat each vertex $v$ of $A$ as a condition on
vertices of $B$, which states that the
vertices not connected to $v$ must be
affinely independent from the
vertices that are connected to $v$. We consider
the
vertices one by one in an increasing order according to the degree. Suppose
that at some step the degree of the vertex considered $v$ from $A$ is $D$,
and the dimension of the sphere on which $B$ currently lies is $d$. If
$D>d$, then we add one dimension and move points not connected to $v$
into that new direction. If $D\le d$, then we do not add the dimension
and rearrange all the points of $B$, so that they now lie in a general
position on the sphere. In both cases the condition is satisfied. We
also keep track of the ``$\eps$-flat'' condition from the remark, moving
vertices into the new direction just a little. This approach allows us
to estimate the total number of edges in the graph needed so that this
algorithm ends with a sphere of dimension at least $d$. At the
last step we embed all the vertices from $A$, and the fact that all the
conditions are satisfied allows us to get exactly the edges needed.\\

We proceed with the detailed proof. We treat vertices $a_{n+1},\ldots,
a_{n+s}$ separately, so we have cases depending on $s$. First we find a
specific realization of the set $B$ on a $k$-dimensional sphere $S_r^{k}$
of radius $r$. The dimension $k$ and radius $r$ will be defined later.
Define the set system $\mathcal
H,\ \mathcal H = \{H_1,\ldots, H_n\},$ where the set $H_i$ is the
subset of indices of vertices from $B$ that are connected to the vertex
$a_i$. Note that some sets may coincide, and that $|H_i|\le |H_j|,\
i\le j.$

Let $X=\{x_1,\ldots, x_m\}$ be a set of points in $\R^d$. We will
denote by $l$-condition the condition $x_i\notin\mathrm{aff}\{x_j:\
j\in H_l\}$ for all $i\notin H_l$. By $\mk H_l$ we denote the set of
all $i$-conditions, $i\le l$. We use the notation $(B, \mk H)$
for the set $B$ of vertices and the set $\mathcal H$ of conditions~
(here we are slightly abusing notation, identifying sets of indices
with the conditions they impose).  We
say that $(B, \mk H)$ is realizable in $S^d,$ if there is a set of
distinct points $X=\{x_1,\ldots, x_m\}\subset S^{d}$ such that $X$
satisfies all the conditions from $\mk H$ and such that $X$ is $\eps$-flat with respect to some plane that passes through the center of $S^d$ for some small fixed $\eps$, say, $\eps = 0.01$.

Choose $k$ to be the minimal dimension such that $(B, \mk H)$ is
realizable in $S^{k}$.


Next, we find a faithful distance realization for $G$.
We consider several cases. \\

$\boldsymbol{s\ge 3:}$\ \  Fix some $0<r<1$ and find a realization of
$(B, \mk H)$ on the sphere $ S^k_r$ described above. We have to find
a proper point $y_i$ for each vertex $a_i$ from $A$. Embed the points
$y_{n+1},\ldots, y_{n+s}$ on the complimentary sphere $S$ of $S_r^k$.
Next, choose $y_i, i\le n,$ on the complimentary sphere $S_i$ to the
minimal sphere that contains the
points $y_j,\ j\in H_i$. It is clear that if
the dimensions of $S, S_i$
are at least 1 (they are at least circles), then,
using a standard general position argument, we can find distinct points
$y_i$ that are at unit distance apart precisely from the points  $x_j,
j\in H_i.$ Indeed, in this case we do not need conditions (b), (c)
at all since all the vertices of $A$ have at least a
one-dimensional sphere as a possible position, and condition (d) is satisfied due to the $\eps$-flatness.
It follows that if $d\ge k+3,$ then we can find the desired
realization.

On the other hand, if $d\le k+2$ then it is clear that there is no
faithful distance realization of $G$ in $\R^d$. Indeed, the points
$y_{n+1},y_{n+2}, y_{n+3}$ are in general position, and all the points
of
$X$ lie on the $(d-3)$-dimensional sphere, complimentary to the circle
circumscribed around $y_{n+1},y_{n+2}, y_{n+3}$. But there is no such
realization of $(B, \mk H)$, since $d-3<k.$\\

$\boldsymbol{s=2:}$\ \  This case is similar to the case $s\ge 3$, with the
only difference that we need $S$ to be zero-dimensional. We additionally require $r<1/2$, so that the diameter of $S$ is bigger than 1, and the condition (c) is satisfied. Condition (b) is again redundant, since we may need it only for the vertices $a_1,\ldots, a_n$, and they have an at least one-dimensional sphere as a possible position. As a result,
we need the following inequality: $d\ge k+2$. This bound is tight for
the same reasons.\\

$\boldsymbol{s=0\text{ {\bf or} } 1:}$\ \  We find a realization of $(B,
\mk H)$ on the sphere $S^k_1$.  If $s=1$, then we place $y_{n+1}$ in the
center of the sphere $S^k_1$. The rest of the points $y_i$ are
placed almost as in the case
$s\ge 3$. We have to make sure that no plane $\mathrm{aff}\{x_j:x_j\in
H_i\}$ contains the center of the sphere $S^k_1,$ otherwise there may be
no room for $y_i$. The existence of such realization again follows from
the general position argument.
Then the conditions (c), (d) are satisfied. As for the condition (b), all the points from $X$ lie on the unit sphere $S^k_1$, and any other unit sphere intersects $S^k_1$ in a hypersphere, and since the affine independence conditions are satisfied, we do not get any extra edges between vertices $a_i$ and $b_j$.

One would expect that in this case we can
find a faithful distance realization of $G$ if $d\ge k+1.$ This is not
exactly the case. If $d=k+1$ and for some $i,j,l$ we have $H_i=H_j=H_l$
and $S_i=S_j=S_l$ consists of two points, then we cannot find room
for all of $a_i, a_j, a_l.$

In this case we have to modify slightly the construction of $X$. We
choose $k$ to be a minimum dimension such that there are points
$X=\{x_1,\ldots, x_m\}\subset S^{k}_1$ that satisfy the conditions
from $\mk H$ and are $\eps$-flat. Suppose there exists a configuration $X$ such that for
all triples  $i_1,i_2,i_3$, for which we have $H_{i_1}=H_{i_2}=H_{i_3}$,
the dimension of the complimentary sphere to the sphere circumscribed
around $x_j,\ j\in H_{i_1}$, is at least one. Then this $X$ is the desired construction, and $G$ is realizable as a
faithful distance graph in $\R^{k+1}$. If not, then $G$ is realizable in
$\R^{k+2}$. This is tight for $s=1$. It is unclear whether this is
is tight for $s=0$ or not, since the points of $B$ need not lie on the
sphere.\\

It seems hard to find the  minimum dimension $k$ in which we can
realize $(B, \mk H)$. But, nevertheless, we can use a simple realization
algorithm that provides a relatively good upper bound on $k$, as describe
next. Consider the conditions one by one and modify the set $X$ so that it
satisfies the conditions that were already considered. Next we describe
the realization of  $(B, \mk H)$ on the $k$-dimensional sphere. Note
that if we find a realization on the sphere of some radius, then, using
homothety, we can change the radius to any desired
prescribed positive value.

\begin{itemize}\item In the zero step we take points $x^0_1,\ldots,
x^0_m$ in general position on the circle. No conditions are
considered at this step.
\item In step $l$ we find such $X^l=\{x^l_j,\
j=1,\dots,m\}\subset S^{k_l}$ that satisfies the conditions $\mk
H_l$. In
this step we get one additional condition ($l$-condition).
  We have two possibilities.

 $\mathbf{|H_l|\ge k_{l-1}+1\ \ }$ If $|H_l|\ge k_{l-1}+1,$ then we put
 $k_l = k_{l-1}+1$ and modify the set $X^{l-1} = \{x^{l-1}_1,\ldots,
 x^{l-1}_m\}$ in the following way. Initially,
the first $(k_l-1)$ coordinates of
 $x^l_i$ are just the coordinates of $x_i^{l-1}$, and we put the last coordinate of $x^l_i$ to be equal to $0$.  If $i\in H_l$, then we rotate the point by the angle equal to $f(l,\eps)$ into that new direction, and if $i\notin
 H_l$, then we rotate the point by the same angle into the opposite direction. In that case
 the $l$-condition is satisfied, and if we choose $|f(l,\eps)|$ decreasing rapidly enough, then the $\eps$-flatness condition is also satisfied.

 $\mathbf{|H_l|\le k_{l-1}\ \ }$ Recall that $|H_i|\le |H_j|$ if $i\le
 j$. If $|H_l|\le k_{l-1},$ then $|H_i|\le k_{l-1}$, where $ i\le l.$
 We put $k_l = k_{l-1}$ and find a set of $m$ points in $S^{k_l}$
 in general position that are $\eps/2$-flat. Then the conditions from $\mk H_l$ are satisfied.
\end{itemize}

Using this algorithm we can estimate how many edges the graph $G$
should have so that $(B, \mk H)$ cannot be realized in $S^{k}$.

\begin{lem}\label{lemedge} If $(B,\mk H),\ \mk H = \{H_1,\ldots, H_n\},$
cannot be realized on the sphere $S^{k}$ then $\sum_{i=1}^n|H_i|\ge
{k+3\choose 2}-3.$\end{lem}

\begin{proof} If $(B,\mk H)$ cannot be realized in the sphere
$S^{k}$, then it cannot be realized in $S^{k}$ using the described
algorithm. The proof is by induction. For $k=1$ we need at least one
set $H_i$ to be of cardinality at least three, so the condition is
satisfied. Consider a pair $(B, \mathcal H)$ that cannot be realized
in $S^{k}$ using the described algorithm. Find the minimum $l$, $l<n$,
such that $(B, \mk H_l)$ cannot be realized in $S^{k-1}.$ Such
$l$ exists since at each step of the algorithm we increase the dimension by
at most 1. By induction, $\sum_{i=1}^l|H_i|\ge {k+2\choose 2}-3.$ Then,
surely, $(B, \mk H_l)$ can be realized in $S^{k}.$ Consequently,
$|H_n|\ge k+2$, otherwise $|H_i|\le k+1,\ i=1,\ldots, n,$ and $(B,
\mk H_l)$ can be realized in $S^{k}.$ Then $\sum_{i=1}^n|H_i|\ge
{k+2\choose 2}-3+(k+2)\ge {k+3\choose 2}-3.$
\end{proof}

Modifying the proof slightly, we can get the following generalization:
\begin{lem}\label{lemedge2} Consider a sequence $|H_1|,\ldots,|H_n|$. Choose a subsequence $i_1<\ldots <i_{s}$ of $1,\ldots, n$ of maximal length with the following properties: $|H_{i_j}|\ge j+2, j=1,\ldots,s$ and each $i_j$ is the minimal number that satisfies this property. Then, if $s\le k-1$, $(B,\mk H)$ is realizable in $S^k$.
\end{lem}

 Next we return to the bipartite graph $G$. We want to estimate
the number of edges $G$ should have so that $G$ is not realizable as
a faithful distance graph in $\R^d$. We again consider several cases
depending on $s$:\\

$\boldsymbol{s\ge 3:}$\ \ In this case $(B, \mk H)$ cannot be realized
on $S^{d-3}_r,$ so by Lemma \ref{lemedge} we have $\sum_{i=1}^n|H_i|\ge
{d\choose 2}-3.$ Moreover, $|H_{n+1}|=\ldots=|H_{n+s}| = m,$ so
$\sum_{i=n+1}^{n+s}|H_i|=sm\ge 3m.$ But, on the other hand, $m\ge d$,
since otherwise the graph $G$ is realizable in $\R^d$ by part 3 of
Theorem \ref{th1}. So $$\sum_{i=1}^{n+s}|H_i|\ge {d\choose 2}-3 +3d=
{d+3\choose 2}-6.$$\\

$\boldsymbol{s= 2:}$\ \ In this case $(B, \mk H)$ cannot be realized
on $S^{d-2}_r,$ so by Lemma \ref{lemedge}  $\sum_{i=1}^n|H_i|\ge
{d+1\choose 2}-3.$ We have $m\ge d+1$ since otherwise there are $m$
points on $S^{d-2}_r$ forming a simplex, and, consequently, satisfying
the conditions $\mk H$. Similarly to the previous case we obtain
 $$\sum_{i=1}^{n+2}|H_i|\ge {d+1\choose 2}-3 +2(d+1)= {d+3\choose
2}-4.$$\\

$\boldsymbol{s=1:}$\ \ In this case we have two possibilities. Assume
first that for any triple  $i_1,i_2,i_3$, for which we have
$H_{i_1}=H_{i_2}=H_{i_3},$ we also have  $|H_{i_3}|< d$. Then for any such
$i_1, i_2,i_3$ and in any realization of the set $B$ in the space $\R^d$
the dimension of the complimentary sphere to the sphere, circumscribed
around $x_j, \ j\in H_{i_1}$ is at least 1. Consequently, $(B, \mk H)$
cannot be realized on $S^{d-1}_1,$ and  by Lemma \ref{lemedge} we have
$\sum_{i=1}^n|H_i|\ge {d+2\choose 2}-3.$ Similarly to the previous case we
obtain that $m\ge d+2$. So $\sum_{i=1}^{n+1}|H_i|\ge {d+3\choose 2}-3.$

Next we assume that there is a triple  $i_1,i_2,i_3$ for which we have
$H_{i_1}=H_{i_2}=H_{i_3}$ and  $|H_{i_3}|\ge d$. Note that the pair $(B,
\mk H)$ can be realized on $S^{f}_r$ if and only if the pair $(B,\mk
H')$ can be realized on $S^{f}_r$, where $\mk H' = \{H_i,\ i=1,\ldots,
n,\ i\neq i_2,i\neq i_3\}.$ On the other hand, $(B, \mk H)$
cannot be realized on $S^{d-2}_r,$ so $\sum_{i=1}^n|H_i|\ge {d+1\choose
2}-3+2|H_{i_2}|.$ Again, $m\ge d+1$. So $\sum_{i=1}^{n+1}|H_i|\ge
{d+4\choose 2}-8.$\\

$\boldsymbol{s=0:}$\ \ We again have two possibilities. If there
is a triple  $i_1,i_2,i_3$ for which we have $H_{i_1}=H_{i_2}=H_{i_3}$
and  $|H_{i_3}|\ge d$, then we obtain $\sum_{i=1}^n|H_i|\ge {d+1\choose
2}-3+2|H_{i_2}|\ge {d+3\choose 2}-6.$

Suppose that for any triple  $i_1,i_2,i_3$, for which we have
$H_{i_1}=H_{i_2}=H_{i_3},$ we also have  $|H_{i_3}|< d$. If we apply
the previous technique directly, we obtain the bound $\sum_{i=1}^n|H_i|\ge
{d+2\choose 2}-3.$ To get a better bound  we modify the realization
algorithm. Note that in this case we have additional flexibility which
is not taken into account by the algorithm: the algorithm produces
a
set $X$ that lies on the sphere, and we do not need it in this case.

Suppose the set $B$ contains a vertex, say $b_1,$ of degree 3, which
is connected to $a_{i_1}, a_{i_2}, a_{i_3}$. Then we exclude $b_1$ out
of  $B$, and apply the usual algorithm for $(B\backslash\{b_1\},\mk H),$
where $\mk H$ is modified in such a way that element $\{1\}$ is excluded
out of its sets. Suppose this pair can be realized on $S^{d-1}_r,$
where $r$ is sufficiently small. Then we try to choose an appropriate
position for the vertices $y_{i_1}, y_{i_2}, y_{i_3}$ so that $y_{i_1},
y_{i_2}, y_{i_3}$ do not lie on one line and form a triangle with a
radius of a circumscribed circle less than one. We surely can
guarantee that $y_{i_1}, y_{i_2}, y_{i_3}$ are in general position, if
at least one of the spheres $S_{i_j}$ (the geometric place of the point
$y_{i_j}$), $j=1,2,3$ is not zero-dimensional. Suppose all of them are
zero-dimensional. It means that $|H_{i_j}|\ge d,\ j=1,2,3$. We apply Lemma \ref{lemedge2} and obtain that
$\sum_{i=1}^n|H_i|\ge {d+2\choose 2}.$ Indeed, there are different
cases, when some of  $i_j, j=1,2,3,$ fall into the maximal
sequence, and in any case it is clear provided that $d\ge 2$ (note
that the optimal bound given in Lemma \ref{lemedge} can only be obtained when the sequence of $|H_i|$ is a progression $3,4,\ldots, d+1$ and when all $|H_i|$ are present).

 We choose points $y_{i_1}, y_{i_2}, y_{i_3}$ in general position
and such that the plane $\mathrm{aff}\{y_{i_1},y_{i_2},y_{i_3}\}$
does not contain the center $O$ of the sphere that contains all
$x_i$. This is possible since the center of the sphere $S_{i_j}$ coincides with the center of the sphere that contains $x_l, l\in H_{i_j},$ while the plane that contains  $x_l, l\in H_{i_j},$ does not contain $O$. So the centers of the spheres $S_{i_j}$ are different from $O$. Then it is not difficult to prove that we can choose $y_{i_1}, y_{i_2}, y_{i_3}$
so that the radius of the circumscribed circle around them is less than
1. If we view $S^{d-1}_r$ as a point, then
the points on $S_{i_j}$ are just some unit vectors going out
of $S^{d-1}_r.$
We can choose a hyperplane $\pi$ that passes through $S^{d-1}_r$
with the following condition: there are affinely independent points
$y_{i_j}\in S_{i_j}$  that lie at distance $\ge c$ apart from $\pi$
and in the same half-space, where  $c>0$ is an absolute constant. Then,
if we move a point $u$ from the sphere $S^{d-1}_r$ orthogonally to
$\pi$ inside the half-space that contains $y_{i_j}$, at some moment the
distance between $u$ and $y_{i_j}$ will be equal to $1-c',$ for any
$c'\le \sqrt{1-c^2}$. Since we can choose $r$ sufficiently small and $c$
can be chosen independently of $r$, we can find $y_{i_1}, y_{i_2},
y_{i_3}$ with the desired properties.

Next we just choose the point $x_1$ in $\R^d$ in such a way
that $|y_{i_j}x|=1,\ j=1,2,3$ and that the sphere of radius 1
with center in $x_1$ does not contain $x_2,\ldots, x_m$ and $S_i, i\neq i_1,i_2,i_3.$
After that we choose appropriate points $y_i, i\neq i_1,i_2,i_3$.
This
means that $G$ is realizable in $\R^d$, a contradiction. Thus, the
pair $(B\backslash\{b_1\},\mk H)$ is not realizable in $\R^d$, and
$\sum_{i=1}^n|H_i|\ge {d+2\choose 2}-3+\deg(b_1)={d+2\choose 2}.$\\

Suppose next that the set $B$ does not contain vertices of degree 3,
but contains a vertex, say $b_1,$ of degree 4, which is connected to
$a_{i_1}, a_{i_2}, a_{i_3}, a_{i_4}$. We argue as in the previous case,
and try to find appropriate $y_{i_1}, y_{i_2}, y_{i_3}, y_{i_4}$ that
are in general position. If there are no such $y_{i_j}$, then we have
two possible reasons for that. The first is that three out of
the spheres $S_{i_j}$
are zero-dimensional, and then we can
interchange the roles of the parts $A$ and
$B$. In this case we get three sets of equal size,
and can conclude as above that $\sum_{i=1}^n|H_i|\ge {d+2\choose
2}.$ The second is that none of the spheres
$S_{i_j}$ are two dimensional, so
we have at least $s$ one-dimensional spheres and $(4-s)$  zero-dimensional spheres, where $s\ge 2$. In any case,
applying Lemma \ref{lemedge2}, we get
that $\sum_{i=1}^n|H_i|\ge {d+2\choose 2}.$

If there are $y_{i_1}, y_{i_2}, y_{i_3}, y_{i_4}$ in general position,
then again one can show that they can be chosen in such a way that the
radius of the circumscribed sphere around them is less than 1, and the
reasoning goes as for the case of a vertex of degree 3. Finally, we get
the estimate $\sum_{i=1}^n|H_i|\ge {d+2\choose 2}-3+\deg(b_1)={d+2\choose
2}+1.$\\

Suppose now that the smallest degree of a vertex in $B$ is 5. Then
we interchange the roles of parts $A$ and $B$, form a set system $\mk
H^B=\{H_1^B,\ldots, H_m^B\}$ analogous to the way we formed the set system
$\mk H$ and apply Lemma \ref{lemedge2}. As the first two elements of the increasing sequence we get
$|H_1^B|,|H_2^B|\ge 5$ instead of $|H_1^B|=3, |H_2^B|=4$, and
we finally get $\sum_{i=1}^n|H_i|=\sum_{i=1}^m|H^B_i|\ge {d+2\choose 2}.$
This completes the proof of Theorem \ref{bipart}. \qquad\qquad\qquad\qquad\qquad\qquad\qquad\qquad\qquad\qquad$\Box$

\subsection{Proof of Proposition \ref{thram}}

Having the statement of Theorem \ref{th2}, the
proof of both parts is merely a slight modification of the proof of the
lower bound for the classical Ramsey number. Indeed, by
a simple probabilistic
argument one can show that if ${m\choose s}2^{1-{s\choose
2}}|\mathcal{FD}_s(d)|<1,$ then $R_{FD}(s,s,d)>m.$

For the proof of part 1 of  Theorem \ref{thram} we use part 2 of
Theorem \ref{th2}, and obtain the inequality $m^s2^{-(1+ o(1))s^2/2}<
1,$ which holds for $m=2^{(1+o(1))s/2}$.
For the proof of part 2 of  Theorem \ref{thram} we use part
3 of Theorem \ref{th2}, and obtain the inequality $m^s2^{-(1/2-c'+
o(1))s^2}< 1.$ Thus, we can choose $\alpha = 1/2-c'$.

\subsection{Proof of Proposition \ref{prgir}}\label{sec25}

We use the following theorem from \cite{Ach}:
\begin{thm}[D. Achlioptas, C. Moore, \cite{Ach}]
\label{achl}
Given any integer $l\ge 3,$ let $k$ be the
smallest integer such that $l\le 2k\log k.$ Then with high
probability the chromatic
number of the random $l$-regular graph is $k,k+1$ or $k+2$.
\end{thm}

We also need a theorem from \cite{Mc}:

\begin{thm}[B.D. McKay, N.C. Wormald, B. Wysocka, \cite{Mc}]
\label{md}
For $(l-1)^{2g-1} =o(n)$, the probability that a
random  $l$-regular graph has girth greater than $g$ is
$$
\exp\left(-\sum_{r=3}^g\frac{(l-1)^r}{2r}+o(1)\right)
$$
\end{thm}

By Theorem \ref{md} we get that for any fixed $l,g\in \N$ the random
$l$-regular graph has girth $\ge g$ with probability bounded away
from $0$. Thus, by
Theorem \ref{achl}, a random $l$-regular graph
satisfies, with positive probability, the condition on the chromatic
number from Theorem \ref{achl} and also has girth greater than $g$. Consider
such a graph $G$ with $l=[d/2].$ Then $\chi(G)=\frac d{4\log d}(1+
o(1)).$
Finally, we use the following theorem from \cite{RM}:
\begin{thm}[H. Maehara, V. R\"odl, \cite{RM}]
\label{thRM}
Any graph with maximum degree $k$ can be realized as a faithful distance
graph in $\R^{2k}$.
\end{thm}

Applying Theorem \ref{thRM} to the graph $G$,
we obtain the statement of Proposition~\ref{prgir}.

\section{Additional Problems}
Problem \ref{pr1} seems to be quite challenging, and is probably
the most interesting question among the ones stated in this paper.

Theorem \ref{bipart} supplies relatively tight bounds on $g_2(d)$, but
it will be interesting to determine the exact value. We believe that the
graph that provides the upper bound in Theorem \ref{bipart} is optimal.

More generally, we suggest the following problem:
\begin{prob}
Determine the
minimum possible number $g_k(d)$ of edges of a $k$-colorable
graph $G$ which is
not realizable as a faithful distance graph in $\R^d$.
\end{prob}

It seems interesting to find any non-trivial examples of graphs
that are not faithful distance graphs in $\R^d$
and  have a small number of
edges. We do not know any example except for bipartite graphs similar to
the one that gives the upper bound in Theorem \ref{bipart}. Is there
any non-trivial example whose number of edges is between that of
$K_{d+2}$ and this upper bound ?

Recall the distance Ramsey numbers discussed in
Section 1. Can we determine
the minimum $f_D=f_D(s)$, such that $R_D(s,s,f_D)=s?$ In other
words, $f_D(s)$ is the smallest possible $d$, such that for any graph $G$
on $s$ vertices either $G$ or its complement $\bar G$ can
be realized as a distance
graph in $\R^d$.

We can show that $f_D(s)=(\frac{1}{2}+o(1)) s$.
The lower bound follows by considering the graph $G$ which is a
clique on $\lceil s/2 \rceil$ vertices (and $\lfloor s/2 \rfloor$
isolated ones.) The upper bound follows from the fact
(proved by an iterative application of the classical Ramsey
theorem) that
the vertices of any graph $G$ on $s$ vertices can be partitioned into
$O(s/ \log s)=o(s)$ pairwise disjoint sets, each spanning either a
clique or an independent set of $G$. This implies that either $G$
or $\bar G$ can be colored properly by $(\frac{1}{2}+o(1))s$ colors
so that at least $s/2$ color classes are of size $1$. The argument
in the proof of Theorem \ref{th1}, part 1 can be easily modified
to show that any graph that has a proper coloring with $a$ color
classes of size $1$ and $b$ bigger color classes can be realized as
a distance graph in $\R^{a+2b}$, implying the desired upper bound.

A similar question can be asked
for the function $f_{FD}(s)$ whose definition is obtained from that
of  $f_D$ by replacing distance Ramsey numbers by faithful distance
Ramsey numbers. It seems harder to determine
the asymptotic behaviour of $f_{FD}(s)$.
We suggest the following

\begin{prob} Determine $f_D(s)$ and $f_{FD}(s)$.
\end{prob}



\end{document}